\documentclass[12pt]{amsart}    
\usepackage{amssymb,amsmath,eepic,verbatim, bm}
\usepackage{amsfonts}
\usepackage{amsbsy}
\usepackage{mathtools}
\usepackage{enumitem}

\usepackage[boxsize=1.2em]{ytableau}

\usepackage{relsize}

\usepackage{pstricks,pst-node}

\psset{unit=1pt}
\psset{arrowsize=4pt 1}
\psset{linewidth=.5pt}

\countdef\x=23
\countdef\y=24
\countdef\z=25
\countdef\t=26

\def\graybox(#1,#2){
\x=#1 \y=#2 
\z=\x \t=\y
\advance\z by 10 
\advance\t by 10 
\psframe[fillstyle=solid,fillcolor=lightgray,linewidth=0pt](\x,\y)(\z,\t) 
\psline[linewidth=.5pt](\x,\y)(\x,\t)(\z,\t)(\z,\y)(\x,\y)}

\def\emptygraybox(#1,#2){
\x=#1 \y=#2 
\z=\x \t=\y
\advance\z by 10 
\advance\t by 10 
\psframe[fillstyle=solid,fillcolor=lightgray,linewidth=0pt,linecolor=lightgray](\x,\y)(\z,\t)}

\def\blankbox(#1,#2){
\x=#1 \y=#2 
\z=\x \t=\y
\advance\z by 10 
\advance\t by 10 
\psline[linewidth=.5pt](\x,\y)(\x,\t)(\z,\t)(\z,\y)(\x,\y)}

\newtheorem{thm}{Theorem}
\newtheorem{lem}[thm]{Lemma}
\newtheorem{prop}[thm]{Proposition}
\newtheorem{cor}[thm]{Corollary}

\numberwithin{thm}{section}

\theoremstyle{definition}

\newtheorem{rmk}[thm]{Remark}

\renewcommand{\P}{\mathbb{P}}
\newcommand{\Z}{\mathbb{Z}}

\newcommand{\on}{\operatorname}

\newcommand{\bull}{ {\scriptscriptstyle{\bullet}}  }

\newcommand{\bG}{\mathbf{\Gamma}}

\renewcommand{\max}{{\text{max}}}

\newcommand{\Pf}{ \on{Pf} }

\newcommand{\rank}{ \on{rank} }

\newcommand{\Det}{ \on{Det} }
\newcommand{\Ker}{ \on{Ker} }

\newcommand{\sgn}{ \on{sgn} }

\newcommand{\wt}{ \on{wt} }

\newcommand{\tenS}{\mathbf{S}}              
\newcommand{\Sgp}{\mathrm{S}} 
\newcommand{\bLambda}{\bm{\Lambda}}    
\newcommand{\bGamma}{\bm{\Gamma}} 
\newcommand{\del}{ \partial }

\let\emph\relax 
\DeclareTextFontCommand{\emph}{\bfseries\em}

\begin{document}
\title{Identities for Schur-type determinants and pfaffians}
\date{March 30, 2021}
\author{David Anderson}
\email{anderson.2804@osu.edu}
\address{Department of Mathematics,
The Ohio State University,
Columbus, Ohio, 43210}

\thanks{DA is partially supported by NSF CAREER DMS-1945212.}

\author{William Fulton}
\email{wfulton@umich.edu}
\address{Department of Mathematics,
University of Michigan,
Ann Arbor, Michigan  48109}

\begin{abstract}
We give a simple formula for some determinants, and an analogous formula for pfaffians, both of which are polynomial identities. The second involve some expressions that interpolate between determinants and pfaffians.  We give several tableau formulas for difference operators in Types A and C, as well as other operators that appear in the ``enriched", or ``back-stable" Schubert polynomials in \cite{AF2}.   There are also tableau formulas for the enriched Schubert polynomials for vexillary and $3 2 1$-avoiding permutations.
\end{abstract}

\maketitle

\setcounter{tocdepth}{2}
\makeatletter
\def\l@subsection{\@tocline{2}{0pt}{4pc}{8pc}{}}
\makeatother

\tableofcontents

\section*{Introduction}  

As far as we know, Kempf and Laksov \cite{KL} were the first to consider Schur-like determinants with rows depending on different variables.  Since then such determinants have appeared as flagged Schur or multi-Schur polynomials, where the sets of variables increase (or decrease) along rows, cf.~\cite{LS}, \cite{W}, \cite{CYY}.  In the first section, we prove a simple identity for such polynomials with no flag conditions on the rows.  To explain the result, we need some notation.

By a \emph{Chern series}  $c$  we mean a power 
series $c = \sum_{k=0}^\infty c_k t^k$, with the $c_k$  in some 
commutative ring and $c_0 = 1$; set $c_k = 0$ for $k < 0$. 
Multiplication of power series makes the Chern series with 
coefficients in a fixed ring into an abelian group, so 
$c = a \cdot b$ means that $c_k = \sum_{i+j=k} a_i b_j$ for all $k$.
 
For any Chern series $c(1), c(2), \ldots, c(n)$, and any partition $\lambda$ of length at most $n$,
define the \emph{Schur determinant} 
\[
S_\lambda(c(1), \ldots, c(n)) \, = \, \Det(c(k)_{\lambda_k+l-k})_{1 \leq k,l \leq n}.
\]
 For a pair of partitions $\mu \subset \lambda$, set
\[
S_{\lambda/\mu}(c(1), \ldots, c(n)) \, = \, \Det(c(k)_{\lambda_k-\mu_l+l-k})_{1 \leq k,l \leq n}.
\]
If all $c(i)$ are equal to the same $c$, we write these as $S_\lambda(c)$ and $S_{\lambda/\mu}(c)$.

Since the $S_\mu(c)$, as $\mu$ varies over all partitions, form an additive basis of the polynomial ring $\Z[c]=\Z[c_1, c_2, \ldots] $, it follows that, for any Chern series $c$ and $a(1), \ldots, a(n)$, and for any partition $\lambda$ of length at most $n$, there are (unique) polynomials $A_{\lambda,\mu}$  in the variables $a(i)_j$ such that 
\[
 S_\lambda(a(1)\cdot c, \ldots,a(n)\cdot c) \, = \, \sum A_{\lambda,\mu} \, S_\mu(c),
\] 
the sum over partitions $\mu$.  In \S1 we show that, in fact, these coefficients are themselves Schur determinants:
\begin{equation*}
 S_\lambda(a(1)\cdot c, \ldots,a(n)\cdot c) \, = \, \sum_{\mu \subset \lambda} S_{\lambda/\mu}(a(1), \ldots, a(n)) \, S_\mu(c).
\end{equation*}
It is surprising, at least to us, that this simple determinantal identity seems to be new in this generality.  It can be deduced directly from the Cauchy-Binet formula, but we give a different proof.  This identity also gives a variation of the Kempf-Laksov formula in geometry.  

In \S\ref{s.tableaux}, we mention some of the classical cases where such determinants have tableau formulas.  A more general tableau formula is proved that in extreme cases gives formulas for multivariate (double) Schur polynomials and for vexillary double Schubert polynomials.  

The tableau formula is applied to the enriched versions of Schubert polynomials in type A (see \cite{AF2}, inspired by \cite{LLS}).  We give a tableau formula for the enriched Schubert polynomials for arbitrary vexillary and $3 2 1$-avoiding permutations (allowing negative as well as positive integers), and we give explicit formulas for the difference, translation, and twisting operators that appear in this story.  

There is a similar, but more complicated, formula when the Schur determinants are replaced by Schur-like pfaffians; the coefficients are polynomials that interpolate between determinants and pfaffians.  In fact, the formulas are actually polynomial identities, which do not require the alternating property of honest pfaffians.  A tableau formula this time involves removing border strips from shifted Young diagrams. This is applied to give a formula for difference operators in type C.

\section{An Algebraic Identity for Schur-like Determinants}
The identity stated in the introduction concerns determinants whose $(i,j)$ entry depends on a Chern polynomial $c(i) = a(i)\cdot c$ depending on the row.  We will prove the natural generalization to one that depends on the column as well as on the row.  Although the identity in the introduction is the main application, this generalization will be used to compute enriched Schubert polynomials of $3 2 1$-avoiding permutations.

\subsection{The determinantal identity}
Fix a positive integer $n$.  A lower case Greek letter stands for a sequence of integers of length $n$,  so $\kappa = (\kappa_1, \ldots, \kappa_n)$.  We write $\kappa \supset \rho$ to mean that, when the sequences are rearranged in decreasing order, the $i^{\text{th}}$ term of the first is at least as large as the $i^{\text{th}}$  term of the second.  We write $\overline{\kappa}$ for the sequence with $\overline{\kappa}_i = \kappa_i - i$;  so $\overline{\kappa}$ is strictly decreasing when $\kappa$ is weakly decreasing.  

Given $\kappa$ and $\rho$, and Chern series $a(1), \ldots, a(n)$,  $c$, and $b(1), \ldots, b(n)$, 
define $S_{\kappa/\rho}(a(\bull) \, c \, b(\bull))$ by the formula
\[ S_{\kappa/\rho}(a(\bull) \, c \,  b(\bull)) = \Det((a(i)\,c\,b(j))_{\overline{\kappa}_i - \overline{\rho}_j})_{1 \leq i, j \leq n} .\]
Set $S_{\kappa/\rho}(a(\bull)) = \Det(a(i)_{\overline{\kappa}_i - \overline{\rho}_j})$  and $S_{\kappa/\rho}'( b(\bull)) = \Det(b(j)_{\overline{\kappa}_i - \overline{\rho}_j})$.

\begin{thm}\label{mainA}  For all sequences $\kappa$ and $\rho$, and Chern series $a(1), \ldots, a(n)$,  $c$, and $b(1), \ldots, b(n)$,
\[ S_{\kappa/\rho}(a(\bull) \, c \,  b(\bull)) = \sum_{\kappa \supset \lambda \supset \mu \supset \rho}
S_{\kappa/\lambda}(a(\bull)) S_{\lambda/\mu}(c) S_{\mu/\rho}'(b(\bull)),\]
the sum over all weakly decreasing sequences $\lambda$ and $\mu$ of length $n$ 
between $\kappa$ and $\rho$. 
\end{thm}

\begin{cor}
For any partitions $\lambda$ and $\nu$ of length at most $n$, and Chern series $a(1), \ldots, a(n)$ and $c$, 
\begin{equation}
S_{\lambda/\nu}(a(1)\,c, \ldots, a(n)\,c)  =  \sum_{\lambda \supset \mu \supset \nu}  S_{\lambda/\mu}(a(1), \ldots, a(n)) \, S_{\mu/\nu}(c),
\end{equation}
the sum over partitions $\mu$ between $\lambda$ and $\nu$.
\end{cor}

In the special case where $\lambda$ is a partition and the $a(k)$ are all equal to some $a$, the corollary recovers a version of a known formula, cf.~\cite{Mac1} \S{\MakeUppercase{\romannumeral 1}} (5.10):
\[
 S_{\lambda/\nu}(a \cdot c) = \sum  S_{\lambda/\mu}(a) \, S_{\mu/\nu}(c).
\]
the sum over all partitions $\mu$ such that $\lambda\supset \mu\supset \nu$. 

The $S_{\lambda/\mu}(c)$ are not independent, but one can write them in terms of the basis $S_\nu(c)$ by using the identity $S_{\lambda/\mu}(c) = \sum c_{\mu \, \nu}^{\lambda} S_\nu(c)$.

\begin{proof}  We must prove the identity
\[  \Det((a(i)\,c\,b(j))_{\overline{\kappa}_i - \overline{\rho}_j}) = \sum_{\kappa \supset \lambda \supset \mu \supset \rho} \Det(a(i)_{\overline{\kappa}_i - \overline{\lambda}_j}) \Det(c_{\overline{\lambda}_i - \overline{\mu}_j}) \Det(b(j)_{\overline{\mu}_i - \overline{\rho}_j}) .\]
Note first that changing the orderings of $\overline{\kappa}$ or $\overline{\rho}$ in either side of this formula gives the same change of sign in the result.  So we may assume $\kappa$ and $\rho$ are weakly decreasing sequences. 

As the shape of this formula suggests, it can be deduced from the general Cauchy-Binet determinantal formula.  We will give another proof, because the same sequence of steps will be used later in the pfaffian setting. 

Consider first the case where $a(i) = 1$ and $b(j) = 1$ for all $i$ and $j$.  In this case, the only $\lambda$ and $\mu$ that give nonzero terms on the right are when $\lambda = \kappa$ and $\mu = \rho$, in which case the identity is a tautology.

To prove the general case, we need the following lemma, which follows from the linearity of determinants in rows and columns.

\begin{lem}\label{Schurlem} Let $\lambda$ and $\mu$ be sequences of length $n$.  
\begin{enumerate} 
\item Let $x$ be any element of the ring.  Fix $1 \leq k \leq n$. Define $\tilde{a}(k)$ to be 
$(1 + x\, t) \, a(k) $ and $\tilde{a}(i) = a(i)$ for $i \neq k$. Define $\tilde{\lambda}_k = \lambda_k - 1$ and $\tilde{\lambda}_i = \lambda_i$ for $i \neq k$.  Then 
\[  S_{\lambda/\mu}(\tilde{a}(\bull)) = S_{\lambda/\mu}(a(\bull)) + x \, S_{\tilde{\lambda}/\mu}(a(\bull)) .\]

\item  
Let $y$ be any element of the ring.  Fix $1 \leq l \leq n$. Define $\tilde{b}(l)$ to be $(1 + y\, t) \, b(l) $ and $\tilde{b}(j) = b(j)$ for $j \neq l$. Define $\tilde{\mu}_l = \mu_l + 1$ and $\tilde{\mu}_j= \mu_j$ for $j \neq l$.  Then 
\[  S_{\lambda/\mu}'(\tilde{b}(\bull)) = S_{\lambda/\mu}'(b(\bull)) + y \, S_{\lambda/\tilde{\mu}}'(b(\bull)). \]

\end{enumerate}
\end{lem}

The proof of the theorem proceeds as follows.  We know the identity when $\kappa$ is a very small, or $\rho$ is a very large sequence, since then both sides vanish; so we can assume the identity holds for smaller $\kappa$ or larger $\rho$.  By the lemma and induction, the identity is true for $\kappa$ and $\rho$ and given $a(1), \ldots, a(n)$ and $b(1), \ldots, b(n)$ if and only if it is true when any $a(k)$ is replaced by $(1+x_{k,p}\,t)\,a(k)$ or any $b(l)$ is replaced by $(1+y_{l,q}\,t)\,b(l)$.  Starting from the case when all $a(k) = 1$ and $b(l) = 1$, we can repeat these replacements until we arrive at the case where each $a(k)_m$ (resp.~$b(l)_m$) is the $m^{\text{th}}$ elementary symmetric polynomial in a large number of independent variables.  Since these are algebraically independent variables, the identity is true in general.
\end{proof}

\subsection{Application to degeneracy loci}

Assume we have vector bundles $F$ and $E$, on a nonsingular variety, and a morphism $\phi \colon F \to E$,  and subbundles and quotient bundles with ranks indicated by subscripts:
\[
F_{q_s}  \hookrightarrow \dots \hookrightarrow F_{q_1} \hookrightarrow F \overset{\phi}{\longrightarrow} E \twoheadrightarrow E_{p_s} 
\twoheadrightarrow  \dots \twoheadrightarrow E_{p_1},
\]
Fix integers $k_1, \ldots, k_s$  satisfying
\[
0 < k_1 < \cdots < k_s  \; \text{ and } l_1 \geq \cdots \geq l_s > 0,  
\]
where $l_i = q_i - p_i + k_i$.  Define $\lambda$ to be the partition, of length $k_s$, with $\lambda_k = l_i$, for $i$ minimal such that $k_i \geq k$.
Consider a degeneracy locus $\Omega$ given by the rank conditions:
\[ 
\rank( F_{q_i} \to E_{p_i} ) \leq p_i - k_i , \;\; 1 \leq i \leq s .
\]  
From \cite{F} we have the formula for the class $[\Omega]$ of $\Omega$, assuming it has the expected codimension $|\lambda|$:  $[\Omega] = S_\lambda(c(1), \ldots, c(k_s))$,  where  $c(k) = c(E_{p_i} - F_{q_i})$, for $i$ minimal with $k_i \geq k$.  Then $c(k) = a(k) \cdot c$, where 
\[
  c = c(E - F) \quad \text{ and } \quad a(k) = c(F/F_{q_i})/c(\Ker(E \to E_{p_i}).
\]
Theorem \ref{mainA} gives the 
\begin{cor}  
The formula for the class of $\Omega$ is
\[
  [\Omega] \, = \,   \sum S_{\lambda/\mu}(a(1), \ldots, a(k_s)) \, S_\mu(c).  
\]
\end{cor}
When all $E_{p_i} = E$, this gives an expansion of the Kempf-Laksov formula \cite{KL}.

\section{Tableau Formulas for Schur-like Determinants}\label{s.tableaux}

Our main goal is to give a tableau formula for the coefficients of the Schur polynomials $S_\mu(c)$ in the expansion of the enriched Schubert polynomial of a vexillary ($2 1 4 3$-avoiding) permutation of the integers.  We will also give a tableau formula for $3 2 1$-avoiding permutations.  Both of these are determinants, which can be expressed as generating functions of non-intersecting paths.

\subsection{Some classical tableau formulas}

There are cases where formulas for the Schur determinants, involving tableaux or non-intersecting paths are known, cf.~\cite{LS}, \cite{W}, \cite{CYY}.  We do not know the most general setting for such formulas, but record one we need here, from \cite{Mac1} \S  {\MakeUppercase{\romannumeral 1}}.5 Ex.~23: 
\begin{prop}\label{Atableaux}
Let $a = \prod_{i=1}^m \frac{1+y_j t}{1-x_i t} $.  For any partitions $\mu \subset \lambda$,
\[S_{\lambda/\mu}(a ) =  \sum_{T} (x,y)^T, \]
the sum over all tableaux $T$ on the skew shape $\lambda/\mu$  with entries 
$1' < 1 < 2' < 2 < \dots < m' < m$, weakly increasing along rows and down columns, with no $k$ repeated in a column and no $k'$ repeated in a row.
\end{prop}

\noindent Here the monomial $(x,y)^T$ is $\prod x_k^{\# \{k \in T\}} \prod y_k^{\#\{ k' \in T\}}$.  For example, the tableau
\[
T = \begin{ytableau}
1'  & 1  & 2' & 3 \\ 
1'  & 2  & 2   \\ 
1 \\
\end{ytableau} 
\]
contributes the monomial $(x,y)^T = x_1^2 x_2^2 x_3 y_1^2 y_2$ to $S_{(4,3,1)}(a)$.

\begin{cor}\label{c.border}
Let $a = \frac{1+y t}{1-x t} $.  For any partition $\lambda$, 
\[
 S_\lambda(a \cdot c) = \sum x^{v(\lambda/\mu)} \,  y^{h(\lambda/\mu)}  \, (x+y)^{k(\lambda/\mu)} \, S_\mu(c),
\]
the sum over all $\mu$ obtained from $\lambda$ by removing a border 
strip (from its Young diagram), with $v(\lambda/\mu)$ the number of vertical lines between border boxes, $h(\lambda/\mu)$ the number of horizontal lines 
between border boxes,  and $k(\lambda/\mu)$ the number of connected components in the border strip.
\end{cor}

For example, since the border strip $(4,3,1)/(2,1)$ has two vertical lines, one horizontal line, and two connected components, the coefficient of $S_{(2,1)}(c)$ in $S_{(4,3,1)}(a\cdot c)$ is $x^2y(x+y)^2$.


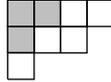
\begin{center}
\pspicture(0,0)(60,60)
\graybox(0,40)
\graybox(10,40)
\blankbox(20,40)
\blankbox(30,40)

\graybox(0,30)
\blankbox(10,30)
\blankbox(20,30)

\blankbox(0,20)

%

\endpspicture
\end{center}

\subsection{The basic tableau formula}\label{ss.tableau}  

The goal of this section is to prove a tableau formula for certain \textit{doubly flagged} Schur polynomials $S_{\lambda/\mu}(a(\bull))$, where the denominators of $a(i)$ get larger, and the numerators get smaller, as $i$ increases.  
We are given two sequences $p_\bull$ and $q_\bull$ of nonnegative integers, each of length $n$, with $p_\bull$ weakly increasing, $q_\bull$ weakly decreasing, with no repeats, i.e., $p_i \leq p_{i+1}$ and $q_i \geq q_{i+1}$ for all $1 \leq i < n$, and at least one of these inequalities is strict for each such $i$.

We set 
\[  a(i) = \frac{\prod_{b=1}^{q_i}(1+y_b\,t)}{\prod_{a=1}^{p_i}(1 - x_a\,t)} .\]
We have a sequence $\lambda$ of length $n$, which satisfies an equation
\[\lambda_i = q_i - p_i + i + t \text{ for some constant } t .\]
The conditions on $p_\bull$ and $a_\bull$ make this sequence weakly decreasing, so it is a partition if $t$ is sufficiently large.  We will give a formula for $S_{\lambda/\mu}(a(\bull))$ for any sequence $\mu \subset \lambda$, i.e., $\mu$ is a weakly decreasing sequence of length $n$ with $\mu_i \leq \lambda_i$ for all $i$.  The skew Young diagram of $\lambda/\mu$ may be obtained by adding a constant to all entries, so they become partitions.  (Skew diagrams that are horizontal translations of each other are identified here; vertical translations are not allowed, as the numbering of their rows is important.)

A \textit{tableau} on $\lambda/\mu$ is a filling of the boxes of its skew Young diagram by plain positive integers $k$, primed positive integers $l'$, or pairs $(k,l')$ of such integers.  In the $i^\text{th}$ row, any $k$ appearing must have $1 \leq k \leq p_i$, and any $l'$ appearing must have $1\leq l \leq q_i$.  The rules for ordering in rows and columns are modifications of the usual rules for plain and primed alphabets.  The plain entries are weakly increasing along rows, strictly increasing down columns, and the primed entries are strictly increasing along rows, weakly increasing down columns.  Any box either right of or below a box containing a plain (resp.~primed) entry must contain a plain (resp.~primed) entry.  In addition, any entry following an entry $(k,l')$ in a row must have the form $(k+m,(l+m+1)')$ for some $m$, and any entry under it must have the form $(k+m+1,(l+m)')$ for some $m$ (with $m \geq 0$ in each case).

\begin{thm}\label{t.tableau}  $S_{\lambda/\mu}(a(\bull)) = \Det(a(k)_{\lambda_k - \mu_l + l - k})_{1 \leq k, l \leq n}$ is the sum of the weights of all tableaux on $\lambda/\mu$, where the weight of a tableau is obtained by setting each $k$ to $x_k$, $l'$ to $y_l$, and $(k,l')$ to $x_k+y_l$.
\end{thm}

We will prove this by showing that the determinant and the tableau formula are given by non-intersecting paths on a directed graph.   We may assume that $t = 0$, since adding the same constant to each $\lambda_i$ and $\mu_i$ does not change the formula. 
Vertices will be at lattice points in the plane, with $(a,b)$  denoting the coordinates.  There will be $n$ \textit{source} points $S_1, \ldots, S_n$  on a vertical line at the left, with the second coordinate of $S_i$ equal to  $\mu_i - i$.  There will be $n$ \textit{target} points $T_1, \ldots, T_n$  on a vertical line at the right, with the second coordinate of $T_i$ equal to $\lambda_i - i $.  

There will be horizontal arrows moving one unit to the right, with weight 1, at any lattice point.   There will be some vertical arrows, moving one unit up, which will have weights $x_k$ or $y_l$ or $x_k+y_l$ for $k$ and $l$ positive.  If the arrow points up from the vertex $(a,b)$, it will get a nonzero weight by the following rules:

\begin{enumerate}
\item $y_{a+b+1}$ if $a \leq 0$, and $a+b \geq 0$;

\item $x_a + y_{a+b+1}$ if $a+b \geq 0$ and $1 \leq a \leq p_i$ and $ a+b \leq q_i $ for some $1 \leq i \leq n$;

\item $x_a$ if $a \geq 1$ and $a+b < 0$.
\end{enumerate} 

Figure 1 shows the case $p_\bull = (0,0,1,3,3,5,6)$ and $q_\bull = (9,7,6,4,2,0,0)$, so $\lambda = (10,9,8,5,4,1,1)$; we take $\mu = (6,4,4,1,0,0,0)$.  A non-intersecting path is shown, and the corresponding tableau is indicated.

\begin{figure}
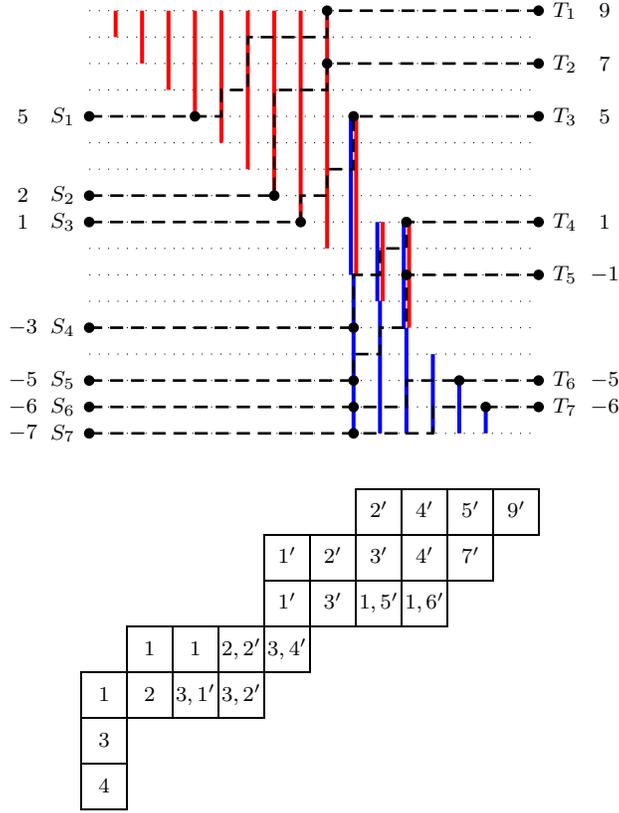


\pspicture(-10,0)(200,200)

\def\cA{red}
\def\cB{blue}
\def\cC{black}

\psline[linestyle=dotted](10,180)(180,180)
\psline[linestyle=dotted](10,170)(180,170)
\psline[linestyle=dotted](10,160)(180,160)
\psline[linestyle=dotted](10,150)(180,150)
\psline[linestyle=dotted](10,140)(180,140)
\psline[linestyle=dotted](10,130)(180,130)
\psline[linestyle=dotted](10,120)(180,120)
\psline[linestyle=dotted](10,110)(180,110)
\psline[linestyle=dotted](10,100)(180,100)
\psline[linestyle=dotted](10,90)(180,90)
\psline[linestyle=dotted](10,80)(180,80)
\psline[linestyle=dotted](10,70)(180,70)
\psline[linestyle=dotted](10,60)(180,60)
\psline[linestyle=dotted](10,50)(180,50)
\psline[linestyle=dotted](10,40)(180,40)
\psline[linestyle=dotted](10,30)(180,30)
\psline[linestyle=dotted](10,20)(180,20)

\psline[linecolor=\cA,linewidth=1.5](20,170)(20,180)
\psline[linecolor=\cA,linewidth=1.5](30,160)(30,180)
\psline[linecolor=\cA,linewidth=1.5](40,150)(40,180)
\psline[linecolor=\cA,linewidth=1.5](50,140)(50,180)
\psline[linecolor=\cA,linewidth=1.5](60,130)(60,180)
\psline[linecolor=\cA,linewidth=1.5](70,120)(70,180)
\psline[linecolor=\cA,linewidth=1.5](80,110)(80,180)
\psline[linecolor=\cA,linewidth=1.5](90,100)(90,180)
\psline[linecolor=\cA,linewidth=1.5](100,90)(100,180)

\psline[linecolor=\cA,linewidth=1.5](111,80)(111,140)
\psline[linecolor=\cB,linewidth=1.5](109,80)(109,140)

\psline[linecolor=\cA,linewidth=1.5](121,70)(121,100)
\psline[linecolor=\cB,linewidth=1.5](119,70)(119,100)

\psline[linecolor=\cA,linewidth=1.5](131,60)(131,100)
\psline[linecolor=\cB,linewidth=1.5](129,60)(129,100)

\psline[linecolor=\cB,linewidth=1.5](110,80)(110,20)
\psline[linecolor=\cB,linewidth=1.5](120,70)(120,20)
\psline[linecolor=\cB,linewidth=1.5](130,60)(130,20)

\psline[linecolor=\cB,linewidth=1.5](140,50)(140,20)
\psline[linecolor=\cB,linewidth=1.5](150,40)(150,20)
\psline[linecolor=\cB,linewidth=1.5](160,30)(160,20)

\pscircle*(10,20){2}
\pscircle*(10,30){2}
\pscircle*(10,40){2}
\pscircle*(10,60){2}
\pscircle*(10,100){2}
\pscircle*(10,110){2}
\pscircle*(10,140){2}

\pscircle*(110,20){2}
\pscircle*(110,30){2}
\pscircle*(110,40){2}
\pscircle*(110,60){2}
\pscircle*(90,100){2}
\pscircle*(80,110){2}
\pscircle*(50,140){2}

\pscircle*(160,30){2}
\pscircle*(150,40){2}
\pscircle*(130,80){2}
\pscircle*(130,100){2}
\pscircle*(110,140){2}
\pscircle*(100,160){2}
\pscircle*(100,180){2}

\pscircle*(180,30){2}
\pscircle*(180,40){2}
\pscircle*(180,80){2}
\pscircle*(180,100){2}
\pscircle*(180,140){2}
\pscircle*(180,160){2}
\pscircle*(180,180){2}

{\tiny
\rput(-15,20){$-7$}
\rput(-15,30){$-6$}
\rput(-15,40){$-5$}
\rput(-15,60){$-3$}
\rput(-15,100){$1$}
\rput(-15,110){$2$}
\rput(-15,140){$5$}
}

{\tiny
\rput(0,20){$S_7$}
\rput(0,30){$S_6$}
\rput(0,40){$S_5$}
\rput(0,60){$S_4$}
\rput(0,100){$S_3$}
\rput(0,110){$S_2$}
\rput(0,140){$S_1$}
}

{\tiny
\rput(190,30){$T_7$}
\rput(190,40){$T_6$}
\rput(190,80){$T_5$}
\rput(190,100){$T_4$}
\rput(190,140){$T_3$}
\rput(190,160){$T_2$}
\rput(190,180){$T_1$}
}

{\tiny
\rput(205,30){$-6$}
\rput(205,40){$-5$}
\rput(205,80){$-1$}
\rput(205,100){$1$}
\rput(205,140){$5$}
\rput(205,160){$7$}
\rput(205,180){$9$}
}

\psline[linewidth=1,linestyle=dashed,linecolor=\cC](10,20)(110,20)(140,20)(140,30)(160,30)(180,30)
\psline[linewidth=1,linestyle=dashed,linecolor=\cC](10,30)(110,30)(130,30)(130,40)(150,40)(180,40)
\psline[linewidth=1,linestyle=dashed,linecolor=\cC](10,40)(110,40)(110,50)(120,50)(120,60)(130,60)(130,80)(180,80)
\psline[linewidth=1,linestyle=dashed,linecolor=\cC](10,60)(110,60)(110,80)(120,80)(120,90)(130,90)(130,100)(180,100)
\psline[linewidth=1,linestyle=dashed,linecolor=\cC](10,100)(90,100)(90,110)(100,110)(100,120)(110,120)(110,140)(180,140)
\psline[linewidth=1,linestyle=dashed,linecolor=\cC](10,110)(80,110)(80,150)(100,150)(100,160)(180,160)
\psline[linewidth=1,linestyle=dashed,linecolor=\cC](10,140)(50,140)(60,140)(60,150)(70,150)(70,170)(100,170)(100,180)(180,180)

\endpspicture

{\tiny
$
\ytableausetup{boxsize=2em}
\begin{ytableau}
\none  & \none  &\none  &\none  &\none  &\none  & 2'  & 4' & 5' & 9' \\ 
\none  & \none  &\none  &\none & 1' & 2' & 3' & 4' & 7'  \\ 
\none  & \none  &\none  &\none & 1' & 3' & 1,5' & 1,6' \\
\none & 1 & 1 & 2,2' & 3,4' \\
1 & 2 & 3,1' & 3,2' \\
3 \\
4
\end{ytableau} 
$
}

\caption{A family of non-intersecting paths and corresponding tableau.  \label{f.network} } 

\end{figure}

For example, the  weight of the third line of the tableau or the third path is 
$y_1y_3(x_1+y_5)(x_1+y_6) $.

The claim is that for each $1 \leq i \leq n$ and each integer $m$, the flow obtained from the directed graph by adding the products of all weights for all paths from $m$ steps below $T_i$ on the left, to the target $T_i$ on the right, is equal to $a(i)_m$, which is a sum of products of elementary symmetric polynomials in $y$ variables times complete symmetric polynomials in $x$ variables.  This is easily verified, by considering each of the three cases where $p_i = 0$, or $q_i = 0$, or $p_i > 0$ and $q_i > 0$.  By the general non-intersecting path Lindstr\"om-Gessel-Viennot result \cite{GV}, this proves  that the sum of the products of all weights, over all non-intersecting paths from the sources $S_i$ to the targets $T_i$ is the required determinant.  And the weights along the $i^{\text{th}}$ path correspond to the entries of the $i^{\text{th}}$ row of the corresponding tableau.

\begin{rmk}
The simplest --- but useful --- case of the theorem is when either $p_i = 0$ or $q_i = 0$ for every $i$.  Then the first rows, where $p_i = 0$, have only primed entries, and the last rows, where $q_i = 0$, have only plain entries; there are no mixed entries $(k,l')$.
\end{rmk}

\begin{rmk} These constructions give positive formulas in some greater generality.  One can take any partition $\lambda$ of 
length $n$ such that
\[  q_{i+1} - p_{i+1} + i + 1 + t \, < \, \lambda_i \, \leq \, q_i - p_i + i + t\]
for $1 \leq i < n$.  The same argument shows that the corresponding determinant will be given by the generating function of the directed graph.  (In fact, $\lambda$ and $\mu$ can be any weakly decreasing sequences of $n$ integers, not necessarily nonnegative.)
\end{rmk}

There is a \emph{dual tableau formula}, where the Chern series depend on the columns instead of the rows.  This time $p_\bull$ is a weakly \textit{decreasing} sequence, and $q_\bull$ is a weakly \textit{increasing} sequence, of $n$ nonnegative integers, with at least one changing at each step.  We are given weakly decreasing sequences $\lambda \supset \mu$ of length $n$, but now $\mu_i = p_i - q_i + i + t$  for some integer $t$.  Set $b(j) = \prod_{b=1}^{q_j}(1+y_b\,t) \, / \, \prod_{a=1}^{p_j}(1 - x_a\,t)$.  We will give a tableau formula for 
\[S_{\lambda/\mu}'(b(\bull)) = \Det(b(j)_{\lambda_i - \mu_j + j - i}).\]
A \textit{dual tableau} satisfies all the conditions  as before (so plain entries in the $i^\text{th}$ row are in $[1,p_i]$ and primed entries are primes of integers in $[1,q_i]$), but all the ordering is \textit{reversed}: the plain/primed entries are weakly/strictly decreasing in rows, strictly/weakly decreasing in columns, and the relations between adjacent pairs is the opposite of the one for tableaux.  The weight of a tableau assigns $x_k$ to $k$, $y_l$ to $l'$, and $x_k+y_l$ to $(k,l')$ as before.

\begin{cor}\label{dualtableau} $S_{\lambda/\mu}'(b(\bull))$  is the sum of the weights of all dual tableaux on $\lambda/\mu$.
\end{cor}
\begin{proof}  Set $\tilde{\lambda}_i = -\mu_{n+1-i}$, $\tilde{\mu}_i = - \lambda_{n_1-i}$, 
$\tilde{p}_i = p_{n+1-i}$,  $\tilde{q}_i = q_{n+1-i}$, and $\tilde{a}(i) = b(n+1-i)$.  Then $S_{\lambda/\mu}'(b(\bull)) = S_{\tilde{\lambda}/\tilde{\mu}}(\tilde{a}(\bull))$.  The skew diagram for $\lambda/\mu$ is the $180^\circ$ rotation of that for $\tilde{\lambda}/\tilde{\mu}$, and the dual tableaux for $\lambda/\mu$ are the $180^\circ$ rotations of the tableau for $\tilde{\lambda}/\tilde{\mu}$.  So the tableau formula for $S_{\tilde{\lambda}/\tilde{\mu}}(\tilde{a}(\bull))$ implies the dual tableau formula for $S_{\lambda/\mu}'(b(\bull))$.
\end{proof}

\subsection{Vexillary permutations of the integers}

A permutation here is a permutation of the integers that fixes all but a finite number of integers; the group of these is denoted $\Sgp_\Z$.  A vexillary permutation is one that avoids the pattern $2 1 4 3$.

A vexillary permutation comes from a triple of integers.  It is convenient to weaken the notion of a triple defined in \cite{AF2}, to allow weak inequalities in the condition to be a triple.  A (weak) \emph{triple} $\tau = (k_\bull, p_\bull, q_\bull)$ of length $s$ has 
\[
0 <  k_1  < k_2 < \cdots  < k_s , \;\;\; p_1 \leq p_2  \leq \cdots  \leq p_s ,  \; \;\; q_1 \geq q_2\geq \cdots \geq q_s ,
\]  
determining integers $l_i$ by the formula $l_i = q_i - p_i + k_i$.  The condition to be a triple is that 
\[
  l_1 \geq  l_2 \geq \cdots \geq l_s  > 0.  \tag{*}  
\]
In this paper a triple where the inequalities in (*) are strict will be called a \emph{strong triple}.  A triple $\tau$ determines a partition $\lambda = \lambda(\tau)$ by setting $\lambda_k = l_i$ for $i$ minimal such that $k_i \geq k$.  And it determines a permutation $w(\tau)$, which is the permutation of minimal length such that for $1 \leq i \leq s$,
$ \# \{a \leq p_i \mid w(a) > q_i\} = k_i$.
A (weak) triple defines the same partition and permutation as the strong triple obtained by omitting any $(k_i,p_i,q_i)$ for which $l_i = l_{i+1}$ (and re-indexing the remaining $(k_j,p_j,q_j)$).  When $\tau$ is strong, the corners of the Young diagram of its partition are at the positions $(k_i,l_i)$, $1 \leq i \leq s$.   We set $n = k_s$, the length of the partition $\lambda$.  Every vexillary permutation comes from a unique strong triple.

A triple is called \emph{complete} if $k_i = i$ for $1 \leq i \leq s = n$.  Any triple can be extended, usually in more than one way, to a complete triple.  Therefore it suffices to consider complete triples, and we will assume from now on that $\tau$ is complete.  Note that in this case $\lambda_i = l_i = q_i - p_i + i$, and $a(i) = a(p_i,q_i)$ for $1 \leq i \leq n$.  A complete triple of length $n$ is determined by the two sequences $p_\bull$ and $q_\bull$ of length $n$, the first weakly increasing, the second weakly decreasing, and the condition (*) is that for every $i < n$, either $p_i < p_{i+1}$ or $q_i > q_{i+1}$ (or both), and $p_n < q_n + n$.  We may write $\tau = (p_\bull,q_\bull)$ for a complete triple, with the understanding that $k_i = i$.

A triple determines Chern series $a(k) = a(p_i,q_i)$ for $i$ minimal such that $k_i \geq k$, where
\[
 a(p,q) \, =  \, \frac{\prod_{i=p+1}^0(1-x_i\,t) \cdot \prod_{j=1}^q (1+y_j\,t)}{\prod_{i=1}^p(1-x_i\,t) \cdot \prod_{j=q+1}^0 (1+y_j\,t)}. 
\]
The enriched Schubert polynomial $\mathbf{S}_{w}$ (cf.~\cite{AF2}) is equal to the determinant
\[ \P_\tau(c,x,y) = S_\lambda(a(\bull)\cdot c) = \sum_{\mu \subset \lambda}S_{\lambda/\mu}(a(\bull)) \, S_\mu(c),\]
where $S_{\lambda/\mu}(a(\bull)) = \Det(a(k)_{\lambda_k - \mu_l + l - k})_{1 \leq k, l \leq n}$. 
This too is the same for weak triple and their associated strong triples; although the entries of the matrices are different, the resulting Schur polynomials agree, by elementary row operations, as in Lemma \ref{Schurlem}.  

The coefficients are polynomials in variables $x_i$, $y_j$,  $\tilde{x}_i = - x_{1-i}$ and $\tilde{y}_j = - y_{1-j}$, for $i$ and $j$ positive.  Our goal is to give a positive formula this determinant in these variables, by showing that it is the sum of weights of tableaux on the skew shape of $\lambda/\mu$.

\subsection{Tableau formula for vexillary determinants}

When $\tau$ is a triple with all $p_i$ and $q_i$ nonnegative, Theorem \ref{t.tableau} gives a formula for the coefficients in $\P_\tau = \sum S_{\lambda/\mu}(a(\bull)) \, S_\mu(c)$.
The general vexillary enriched Schubert polynomials can be realized as specializations of those coming from those using nonnegative or nonpositive sequences, which will give tableau formulas for them. 

Given a complete triple $\tau = (p_\bull,q_\bull)$, 
define integers $0 \leq r \leq s \leq n+1$ by the rules: 
\[ r \, \text{ is the maximal integer such that } p_r < 0 \text{ and } q_r > 0 ,\]
with $r = 0$ if there is no such integer; and
 \[ s  \, \text{ is the minimal integer such that } p_s > 0 \text{ and } q_s < 0,\]
with $s = n+1$ if there is no such integer.  There are two cases:
either (1) $p_i \geq 0$ and $q_i \geq 0$ for all $r < i < s$; or (2) $p_i \leq 0$ and $q_i \leq 0$ for all $r < i < s$.  (The cases where $s = r+1$ or where $s=r+2$ and $p_{r+1} = q_{r+1} = 0$ can be regarded in either case.)

Define the complete triple $\overline{\tau} = (\overline{p}_\bull,\overline{q}_\bull)$ by 
setting 
\[ \overline{p}_i = \begin{cases} 0  \text{ if } i \leq r \\ p_i  \text{ if } r < i < s \\ p_i - q_i  \text{ if } s \leq i \end{cases}  \; \text{ and } \; \overline{q}_i = \begin{cases} q_i - p_i  \text{ if } i \leq r \\ q_i  \text{ if } r < i < s \\ 0  \text{ if } s \leq i \end{cases} \]
So $\overline{\tau}$ is nonnegative in case (1), and nonpositive in case (1).  Note that $\lambda(\overline{\tau}) = \lambda(\tau)$ in either case.

Consider case (1).  Set $\overline{a}(i) = \frac{\prod_{b=1}^{\overline{q}_i}(1 + v_b\,t)}{\prod_{a=1}^{\overline{p}_i}(1-u_a\,t)}$, for variables $u_k$ and $v_l$.  Theorem \ref{t.tableau} gives a tableau formula for $S_{\lambda/\mu}(\overline{a}(\bull))$ as a sum of products of the variables $u_k$, $v_l$, and $u_k + v_l$.  From the definition as determinants, one sees that $S_{\lambda/\mu}(a(\bull))$ can be obtained from $S_{\lambda/\mu}(\overline{a}(\bull))$ by specializing the $u$ and $v$ variables  to the $x$ and $y$ variables by the following rules:

Specialize the variables $v_1, \ldots, v_{\lambda_1 - 1}$ to the variables $y_i$ and $\tilde{x}_j$, by using the orderings $q_r \leq \cdots \leq q_1$  and $-p_r \leq \cdots \leq -p_1$.  That is,  $v_1, \ldots v_{q_r}$ are specialized to $y_1, \ldots, y_{q_r}$; then the next $-p_r$ of the $v$ variables are specialized to 
$\tilde{x}_1, \ldots, \tilde{x}_{-p_r}$; then the the next $q_{r-1} - q_r$ $v$ variables are specialized to $y_{q_r+1}, \ldots, y_{q_{r-1}}$, the next $p_r - p_{r-1}$ to $\tilde{x}_{-p_r+1}, \ldots, \tilde{x}_{-p_{r-1}}$, and so on until all $q_1 - p_1$ of the $v$ variables have been specialized.

Specialize the variables $u_1, \ldots, u_{n - \lambda_n}$ to the variables $x_i$ and $\tilde{y}_j$, by using the orderings $p_s \leq \cdots \leq p_n$  and $-q_s \leq \cdots \leq -q_n$.  That is,  $u_1, \ldots u_{p_s}$ are specialized to $x_1, \ldots, x_{p_s}$, then the next $-q_s$ $u$ variables are specialized to $\tilde{y}_1, \ldots \tilde{y}_{-q_s}$, and so on, alternating, until all $p_n - q_n$ of the $u$ variables are specialized.

Note that for $r < i < s$ the first $p_i$ of the $u$ variables are specialized to $x_1, \ldots, x_{p_i}$ and the first $q_i$ of the $v$ variables are specialized to $y_1, \ldots, y_{q_i}$.

Case (2) is entirely similar, with substitutions as follows.  Specialize the $v$ variables to the $\tilde{x}$'s and $y$'s, using the orderings $-p_r \leq \cdots \leq -p_1$ and $q_r \leq \cdots \leq q_1$.  Specialize the $u$ variables to the $\tilde{y}$'s and $x$'s, using the orderings $-q_s \leq \cdots \leq -q_n$ and  $p_s \leq \cdots \leq p_n$.

The next theorem then follows from Theorem~\ref{t.tableau}.

\begin{thm}  With these substitutions, 
\[S_{\lambda/\mu}(a(1), \ldots, a(n)) = \Det(a(i)_{\lambda_i - \mu_j + j - i})_{1 \leq i,j \leq n} \]
 is the sum of the weights of all tableaux on $\lambda/\mu$.
\end{thm}

This theorem applies to \emph{multivariate Schur polynomials}, sometimes denoted $\Psi_\lambda(c,y)$, for any partition $\lambda$.  In fact, if the Young diagram of $\lambda$ has corners at $(k_i,l_i)$, then 
\[\Psi_\lambda(c,y) = \P_\tau(c,0,y) = \sum S_{\lambda/\mu}(a(\bull)) \,S_\mu(c), \]
 where $\tau = (k_\bull,p_\bull,q_\bull)$, with $p_i = 0$ and $q_i = l_i - k_i$, so 
 \[a(k) =  \frac{\prod_{j=1}^{q_i} (1+y_j\,t)} {\prod_{j=q_i+1}^0 (1+y_j\,t)}, \]
 for $i$ minimal such that $k_i \geq k$.  Note that in this case the tableaux have primed integers above the diagonal, and plain integers below the diagonal, with the usual orderings.  The formula of the theorem recovers a tableau formula of Molev \cite[Cor.~5.6]{Mo}:
 
 \begin{cor}  The coefficient $S_{\lambda/\mu}(a(1), \ldots, a(n))$ is the sum of the weights of all tableaux on $\lambda/\mu$, with weights determined by assigning the variable $y_l$ to primed integers $l'$ and $\tilde{y}_k = -y_{1-k}$ to plain integers $k$. 
\end{cor}

\subsection{A tableau formula for $3 2 1$-avoiding permutations}

We give a tableau formula for the enriched Schubert polynomial of a $3 2 1$-avoiding permutation $w$ in $\Sgp_\Z$.  Such a $w$ is determined by (and determines) two strictly decreasing sequences $p_\bull$ and $q_\bull$ of integers of the same length $n$, with $p_i \leq q_i$ for all $i$;  $w$ is the permutation of minimal length such that $w(p_i) = q_i+1$ for all $i$.  As we saw in \cite{AF2}, this Schubert polynomial $\tenS_w(c,x,y)$ is $\Det((a(i)\cdot c \cdot b(j))_{q_i+1 - p_j})$, where 
\[a(i) = \frac{\prod_{b=1}^{q_i }(1+y_b\,t)}{\prod_{b=q_i+1}^{0}(1+y_b\,t)} \;\; \text{ and } \;\; b(j) = \frac{\prod_{a=p_j+1}^{0}(1 - x_a\,t)}{\prod_{a=1}^{p_j}(1 - x_a\,t)}.\]

Define weakly decreasing sequences $\kappa$ and $\rho$ of length $n$ by setting $\kappa_i = q_i+1+ i$ and $\rho_i = p_i + i$ for all $i$.  By Theorem~\ref{mainA}, 
\[\tenS_w(c,x,y) = \sum_{\kappa \supset \lambda \supset \nu \supset \rho}
S_{\kappa/\lambda}(a(\bull)) S_{\lambda/\nu}(c) S_{\nu/\rho}'(b(\bull)).\]
We will obtain a positive formula for these polynomials by giving tableau formulas for the first and third terms in these sums.  

For the first, take $r$ maximal so $q_r \geq 0$; so $a(i) = \prod_{b=1}^{q_i } (1 + y_b\,t)$ for $i \leq r$ and $a(i) = 1/\prod_{a=1}^{-q_i} (1 - \tilde{y}_a \, t)$ for $i > r$, where $\tilde{y}_a = - y_{1-a}$.  Apply Theorem~\ref{t.tableau} to the sequences $(0,\ldots, 0, -q_{r+1}, \ldots, -q_n)$ and $(q_1, \ldots, q_r, 0, \ldots, 0)$ and the two sets of variables $\tilde{y}_a$ and $y_b$.  Then $S_{\kappa/\lambda}(a(\bull))$ is the sum of the weights of the tableaux, where $k$ is given weight $\tilde{y}_k$ and $l'$ is given weight $y_l$.  Note that there are no pairs $(k,l')$ in these tableaux.  

The third terms use the dual tableau formula.  Take $r$ maximal so $p_r \geq 0$, so $b(j) = 1/\prod_{a=1}^{p_j}(1-x_a\,t)$ if $j \geq r$, and $b(j) = \prod_{b=1}^{-p_j}(1+\tilde{x}_b \,t)$ if $j < r$, where $\tilde{x}_b = -x_{1-b}$.  Use the sequences $(0, \ldots, 0,p_r, \ldots, p_n)$ and $(-p_1, \ldots, -p_r, 0, \ldots, 0)$ and the variables $x_a$ and $\tilde{x}_b$.  Then $S_{\nu/\rho}'(b(\bull))$ is the sum of the weights of the dual tableaux on $\nu/\rho$, where $k$ is given weight $x_k$ and $l'$ is given weight $\tilde{x}_l$.

For these sequences $\lambda \supset \nu$, one has the usual positive decomposition 
\[S_{\lambda/\nu} = \sum \sum c_{\mu \, \nu}^\lambda S_\mu(c),\]
a sum over partitions $\mu$;  here $c_{\mu \, \nu}^\lambda = c_{\mu \, \nu+t}^{\lambda+t}$, for any integer $t$ such that $\nu + t = (\nu_1+t, \ldots, \nu_n + t)$, and 
so $\lambda + t = (\lambda_1+t, \ldots, \lambda_n + t)$ are partitions.

\begin{cor}\label{321} The $3 2 1$-avoiding enriched Schubert polynomial is given by the formula
\[ 
S_w = \sum c_{\mu \, \nu}^\lambda \wt(T) \, S_\mu(c) ,
\]
the sum over partitions $\mu$ and all $\lambda$ and $\nu$ such that $\kappa \supset \lambda \supset \nu \supset \rho$, with the tableaux $T$ on $\kappa/\lambda$ and $\nu/\rho$ as above.  
\end{cor}

In particular, this gives a tableau formula for the classical double Schubert polynomial $\mathfrak{S}_w(x,-y)$ when $w$ is $3 2 1$-avoiding and in $S_{+}$, i.e., the sequences $p_i$ and $q_i$ are positive.  The sum is then over all $\lambda$ between $\rho$ and $\kappa$, taking $\nu = \lambda$.  The resulting formula is different from the one in \cite{CYY} (which weights tableaux by products of binomials, rather than monomials).

\section{Operators in Type A}

In \cite{AF2}, we described a construction of enriched (twisted) Schubert polynomials in type A, which live in a polynomial ring $\bLambda[x,y]$, where  $\bLambda = \Z[c,z]$.  Here $c$ stands for variables $c_k$ of degree $k$, for $k \geq 1$, $z$ is a variable of degree $1$, and $x$ and $y$ stand for variables $x_i$ and $y_i$ for all integers $i$.  (For the untwisted case, set $z$ to be $0$.)  The Schur polynomials  $S_\lambda(c)$ form an additive basis for $\bLambda$ over $\Z[z]$.  We give formulas for three operators on these algebras.

\subsection{Difference operators}

We use the first two sections to give a formula for the difference operator $\del_0 = \del_0^x$ in type A,.  (There is an operator  $\del_0^y$ 
which is obtained by interchanging each $x_i$ with $y_i$; and there are operators $\del_i^x$ and $\del_i^y$ for $i \neq 0$, described in \cite{AF2} and  \S4 and \S10 of \cite{LLS}, which are simpler.)  These are endomorphisms of the $\Z[z]$-algebra $\bLambda[x,y]$.  This $\del_0$ comes with a related automorphism $s_0$, which in this case is given by the formula
\[s_0(c) = c \cdot \frac{1+(z-x_0) t}{1+(z-x_1) t}, \]
where $c = 1 + \sum_{k>0}c_k t^k$.
Applying Theorem~\ref{mainA} and Corollary~\ref{c.border}, with $x = x_1 - z$ and $y = z - x_0$, this gives the formula 
\[
s_0(S_\lambda(c)) \, = 
\, \sum_{\mu \subset \lambda} (x_1-z)^{v(\lambda/\mu)} \, (z-x_0)^{h(\lambda/\mu)} \, (x_1 - x_0)^{k(\lambda/\mu)} \, S_\mu(c).
\]

Since $\del_0(f) = \frac{f - s_0(f)}{x_0-x_1}$, we get

\begin{prop}
\[
\del_0(S_\lambda(c)) \, = 
\, \sum_{\mu \subsetneqq \lambda} (x_1-z)^{v(\lambda/\mu)} \, (z-x_0)^{h(\lambda/\mu)} \, (x_1 - x_0)^{k(\lambda/\mu)-1} \, S_\mu(c).
\]
\end{prop}

Setting $z = 0$ gives the corresponding difference operator  on $\Lambda[x,y]$, with $\Lambda = \Z[c]$:
\[
\del_0(S_\lambda(c)) \, = \, \sum x_1^{v(\lambda/\mu)} \, (-x_0)^{h(\lambda/\mu)} \, (x_1 - x_0)^{k(\lambda/\mu)-1} \, S_\mu(c).
\]

\subsection{Translation operators}

Similarly, we have a formula for the translation operators $\gamma^m$ on $\bLambda[x,y]$, also discussed in \cite{AF2}.  These are defined by the equations 
$\gamma^m(x_i) = x_{m+i}$  and $\gamma^m(y_i) = y_{m+i}$  
for all integers $m$ and $i$, and
\[ 
\gamma^m(c) = c \cdot \prod_{i=1}^m \frac{1+y_i \, t}{1+(z-x_i)\, t}
 \cdot \prod_{i=m+1}^0 \frac{1+(z-x_i)\, t}{1+y_i \, t}. 
\]
(This formula is an abbreviation for two formulas, with the first product used if $m$ is positive, the second if $m$ is negative.)

\begin{prop}
For $m > 0$, 
\[
\gamma^m(S_\lambda(c)) = \sum_{\mu \subset \lambda} \sum_{T} (x-z,y)^T \, S_\mu(c), 
\]
the second sum over all tableaux $T$ on the skew shape $\lambda/\mu$  with entries 
$1' < 1 < 2' < 2 < \dots < m' < m$, weakly increasing along rows and down columns, with no $k$ 
repeated in a column and no $k'$ repeated in a row; and $(x-z,y)^T = \prod (x_k-z)^{\# \{k \in T\}} \prod y_k^{\#\{ k' \in T\}}$.
\end{prop}
The formula for $\gamma^{-m}(S_\lambda(c))$ is similar, but each $k$ in a tableau contributes a factor $-y_k$ and each $k'$ contributes $z - x_k$.  

\subsection{Twisting operator}

We give a tableau formula for the operator that takes the Chern class of a virtual bundle of rank $0$ to the Chern class of its twist (tensor product) by a line bundle.  
We work in the polynomial ring $\bLambda = \Z[c,z]$.

There is an endomorphism $\theta$ of the $\Z[z]$-algebra $\bLambda$ that takes $c_k$ to $\sum_{i=1}^k \tbinom{k-1}{i-1}z^{k-i}c_i$.  This corresponds to twisting by a line bundle in geometry: if $\xi$ is a virtual vector bundle of rank $0$, and $L$ is a line bundle with first Chern class $z$, then $\theta$ takes $c_k(\xi)$ to $c_k(\xi \cdot [L^*])$.

\begin{lem}
For a partition $\lambda$ of length $n$,
\[  \theta(S_\lambda(c)) = \Det(c(i)_{\lambda_i + j - i})_{1 \leq i,j \leq n},\]
where $c(i) = a(i) \cdot c$, with $a(i) = (1+z)^{\lambda_i - i}$ for $1 \leq i \leq n$.
\end{lem}
\begin{proof}  Let $A = (a_{i\,j})$, where $ a_{i\,j} = \theta(c_{\lambda_i+j-i})$.  A succession of column operations will turn $A$ into the matrix appearing on the right side of the identity, so they will have the same determinant.  Each step will add $z$ times the $i^{\text{th}}$ column to the column to its right.  The order of these $i$'s is from  $i = n-1$ down to $i=k$, successively for $k$ varying from $1$ to $n-1$.
\end{proof}

Corollary \ref{c.border} gives a tableau formula for these coefficients.  The maximal $d$ such that $\lambda_d \geq d$ is called the \textit{Durfee square} of $\lambda$, and denoted $d(\lambda)$.  If $\mu \subset \lambda$ have the same Durfee square, define a \emph{tableau} on $\lambda/\mu$ to be a filling of the boxes of this skew shape with positive integers, satisfying:
\begin{enumerate}
\item the entries above the diagonal are strictly increasing across rows, weakly increasing down columns, and the entries in the $i^\text{th}$ row are at most $\lambda_i - i$;
\item the entries below the diagonal are strictly increasing down columns, weakly increasing across rows, and the entries in the $i^\text{th}$ column  are at most $\lambda'_i - i$.
\end{enumerate} 
Let $N(\lambda/\mu)$ be the number of tableaux on $\lambda/\mu$.

For $\mu \subset \lambda$ of the same Durfee square $d$, set 
\[  
z^{\lambda/\mu} \, = \, z^{\sum_{i=1}^d \lambda_i - \mu_i} \, (-z)^{\sum_{i=1}^d \lambda'_i - \mu'_i}.
\]
(The two exponents are the number of boxes above the diagonal and below the diagonal in the skew shape.)

\begin{prop} For any partition $\lambda$ 
\[ \theta(S_\lambda(c)) = \sum_{\substack{\mu \subset \lambda \\d(\mu) = d(\lambda)}} N(\lambda/\mu) \,  z^{\lambda/\mu} \, S_\mu(c).\]
\end{prop}

Using Frobenius notation for partitions, if $\lambda = (\alpha | \beta)$ and $\mu = (\alpha' | \beta')$, where each of $\alpha$, $\beta$, $\alpha'$, and $\beta'$ are strictly decreasing sequences of nonnegative integers of length $d = d(\lambda) = d(\mu)$, the 
coefficient $N(\lambda/\mu)$ is a product of binomial determinants (cf.~\cite{GV}):
\[ N(\lambda/\mu) = \tbinom{\alpha}{\alpha'} \cdot \tbinom{\beta}{\beta'} ,\]
where $\binom{\alpha}{\alpha'} = \Det(\binom{\alpha_i}{\alpha'_j})_{1 \leq i,j \leq d}$ and
$\binom{\beta}{\beta'} = \Det(\binom{\beta_i}{\beta'_j})_{1 \leq i,j \leq d}$.

\section{An Algebraic Identity for Schur-like Pfaffians}
 
We define \emph{pseudo-pfaffians} $\Pf_\lambda(c(1), \ldots, c(n))$ for any sequence $\lambda = (\lambda_1, \ldots, \lambda_n)$ of integers, and any Chern series $c(1), \ldots, c(n)$.  (The ``pseudo" is because no alternating properties are required; the polynomials need not vanish even if $\lambda$ is a partition with some repeated parts.)  To allow a twist, we work in any commutative $\Z[z]$-algebra.  For $n = 1$, $\Pf_p(c) = c_p$.  For $n = 2$,
\[ 
\Pf_{p,q}(a,b) = \sum_{0 \leq i \leq j \leq q} (-1)^j ( \tbinom j i +\tbinom {j-1} i ) z^i \, a_{p+j-i} \, b_{q-j} .
\]
 For $n = 3$, $\Pf_{p,q,r}(a,b,c)$ is defined to be
\[ 
 \Pf_p(a) \, \Pf_{q,r}(b,c)  - \Pf_q(b) \, \Pf_{p,r}(a,c)  + \Pf_r(c) \, \Pf_{p,q}(a,b) .
\]
For $n = 4$, $ \Pf_{p,q,r,s}(a,b,c,d)$ is
\[ 
\Pf_{p,q}(a,b) \,\Pf_{r,s}(c,d)  - \Pf_{p,r}(a,c) \,\Pf_{q,s}(b,d) + \Pf_{p,s}(a,d) \,\Pf_{q,r}(b,d) .
\]
Continuing inductively, for $n$ odd, $\Pf_\lambda(c(1), \ldots, c(n)) $ is
\[
\sum_{k=1}^n (-1)^{k-1} \Pf_{\lambda_k}(c(k)) \Pf_{\lambda_1, \ldots, \widehat{\lambda_k}, \ldots, \lambda_n}(c(1), \ldots, \widehat{c(k)}, \ldots, c(n)),
\]
and for $n$ even, $\Pf_\lambda(c(1), \ldots, c(n)) $ is
\[ 
  \sum_{k=2}^n (-1)^{k} \Pf_{\lambda_1,\lambda_k}(c(1),c(k)) \Pf_{\lambda_2, \ldots, \widehat{\lambda_k}, \ldots, \lambda_n}(c(2), \ldots, \widehat{c(k)}, \ldots, c(n)).
\]
 
We write $\Pf_\lambda(c(\bull))$ for $\Pf_\lambda(c(1), \ldots, c(n))$, and, if all $c(k) = c$ are equal, we write $\Pf_\lambda(c)$ for $\Pf_\lambda(c, \ldots, c)$. 
Our goal here is to express $\Pf_\lambda(a(1)c, \ldots, a(n)c)$ as a sum of polynomials in the $a(i)_j$'s times the pseudo-pfaffians $\Pf_\mu(c)$, as $\mu$ varies over partitions of length at most $n$, for any sequence $\lambda$ of length $n$ and Chern series $a(1), \ldots, a(n)$ and $c$.  If we are in a ring where $\Pf_\mu(c)$ vanishes whenever $\mu$ has repeated entries, which is the case in the basic algebra used for type C Schubert polynomials, this will give a formula in terms of the elements $\Pf_\mu(c)$, as $\mu$ varies over strict partitions.
 
Each coefficient will be a kind of \emph{skew pseudo-pfaffian}, which we denote by 
 $\Pf_{\lambda/\mu}(a(\bull)) = \Pf_{\lambda/\mu}(a(1), \ldots, a(n))$.   These are polynomials that interpolate between pseudo-pfaffians and determinants.  
 They are defined as follows.  Temporarily let $\mu = (\mu_1, \ldots, \mu_n)$ be any sequence of  $n$ nonnegative integers.  For each $k$ that appears in $\mu$, order the set 
\[
 \mu[k] = \{ i \in [1,n] \mid \mu_i = k \} = \{ i_1 < \dots < i_s \}. 
\]
Set 
\[
P[\mu,k] = \Pf_{\lambda_{i_1}-k, \ldots, \lambda_{i_s}-k}(c(i_1), \ldots, c(i_s)) ,  
\]
and let $P[\mu]$ be the product of all $P[\mu,k]$, as $k$ varies over the entries of $\mu$.
 
Now for a partition $\mu$ of length at most $n$, we consider all the {\it distinct} permutations $\sigma$ of its entries.  For example, with $n = 4$,  $\mu = (2,0,0,0)$ has four of them, $\mu = (3,3,0,0)$ has six, and $\mu = (6,4,2,0)$ has $24$. 
Each such permutation has a sign $\sgn(\sigma)$, the minimal number of transpositions needed to get it back in order.  Define
\[ 
 \Pf_{\lambda/\mu}(a(1), \ldots, a(n)) = \sum_{\sigma} \sgn(\sigma) P[\sigma(\mu)], 
\]
the sum over the distinct permutations $\sigma$ of $\mu$.  Note the extreme cases:
\begin{enumerate}
\item If $\mu = (k, \ldots, k)$, then $\Pf_{\lambda/\mu}(a(\bull)) = \Pf_{\lambda_1-k,\ldots, \lambda_n - k}(a(\bull))$. 
\item  If $\mu_1 > \cdots > \mu_n \geq 0$, then $\Pf_{\lambda/\mu}(a(\bull)) = S_{\tilde{\lambda}/\tilde{\mu}}(a(\bull))$, 
 where $\tilde{\lambda}_i = \lambda_i + i - 1$ and $\tilde{\mu}_i = \mu_i + i - 1$.  
\end{enumerate}
If $\lambda$ and $\mu$ are strict partitions, then $\tilde{\lambda}$ and $\tilde{\mu}$ are ordinary partitions, and the shifted skew shape $\lambda/\mu$ is equal to the unshifted shape $\tilde{\lambda}/\tilde{\mu}$.
 
\setcounter{thm}{4}
\begin{thm}
For any sequence $\lambda = (\lambda_1, \ldots, \lambda_n)$ of integers, partition $\nu$ of length at most $n$, and Chern series $a(1), \ldots, a(n)$ and $c$, 
\[  
\Pf_{\lambda/\nu}(a(1)\, c, \ldots, a(n)\, c)  =  \sum  \Pf_{\lambda/\mu}(a(1), \ldots, a(n)) \, \Pf_{\mu/\nu}(c),
\]
the sum over all partitions $\mu \supset \nu$ of length at most $n$ and with 
$ \mu_1 \leq \max(\lambda_1,\ldots, \lambda_n)$.
\end{thm}
If $\lambda$ is a partition, the sum is over partitions $\mu$ with   $\lambda \supset \mu \supset \nu$.  

\begin{cor} 
For any Chern series $a(1), \ldots, , a(n)$ and $c$, and  partition $\lambda$ of length $n$, 
\[
\Pf_{\lambda}(a(1)\,c, \ldots, a(n)\,c) = \sum  \Pf_{\lambda/\mu}(a(1), \ldots, a(n)) \Pf_\mu(c) ,
\]
the sum over all partitions $\mu \subset \lambda$.
\end{cor}
 
\begin{cor}\label{strict}  
If $\Pf_{\mu}(c) = 0$  for every partition $\mu$ that is not strict, then for any partition $\lambda$ of length $n$,
\[
 \Pf_{\lambda}(a(1)\,c, \ldots, a(n)\,c) = \sum  S_{\tilde{\lambda}/\tilde{\mu}}(a(1), \ldots, a(n)) \Pf_\mu(c),
\]
 the sum over strict partitions $\mu$ contained in $\lambda$.
\end{cor}

For the proof, we need two lemmas, the first of which is a straightforward calculation from the definition.
\begin{lem}  
Suppose $p = q+m$, for $m \geq 0$.  Then for any Chern series $c$,
\[ 
\Pf_{p,q}(c) + \Pf_{q,p}(c) = \sum_{k=0}^{\lfloor m/2 \rfloor} (-1)^{m-k}  ( \tbinom{m-k}{m-2k}+\tbinom {m-k-1}{m-2k} ) \Pf_{q+k,q+k}(c) .
\]
\end{lem}

\begin{lem}  
For any sequence $\lambda$ of length $n$ and partition $\mu$ of length at most $n$, and any $1 \leq k \leq n$ and element $x$,
\[ 
\begin{aligned} \Pf_{\lambda/\mu}&(a(1), \ldots, (1 + x\,t)\,a(k), \dots ,a(n)) =  \\
& \Pf_{\lambda/\mu}(a(1), \ldots, a(n)) \, +x \, \Pf_{(\lambda_1,\ldots, \lambda_k - 1,\ldots, \lambda_n)/\mu}(a(1), \ldots, a(n)) .
\end{aligned}
\]
\end{lem}

\begin{proof}
This is evident when $n = 1$.  For $n=2$, it follows from the definition that 
\[ 
\begin{aligned} &\Pf_{p,q}((1+x\,t)a,b) = \Pf_{p,q}(a,b)+ x \Pf_{p-1,q}(a,b), \; \text{ and }\\ 
&\Pf_{p,q}(a,(1+x\,t)b)+\Pf_{p,q}(a,b)+ x \Pf_{p,q-1}(a,b) .
\end{aligned} 
\]
The general case follows from this and the inductive definitions of pseudo-pfaffians.
\end{proof}

The proof of the theorem is similar to the corresponding theorem for Schur-like determinants.  First consider 
the case when $a(k) = 1$ for all $k$.   In this case, if $\lambda$ is a partition, the only $\mu$ for which 
 $\Pf_{\lambda/\mu}(1)$ is not zero is when $\mu = \lambda$, and then it is $1$.  If $\lambda$ is not a partition, 
 apply the first lemma to two consecutive entries $q = \lambda_i < p = \lambda_{i+1}$.  This writes all terms 
 on both sides of the equation to be proved -- in the same way -- as a sum of terms with $\lambda$ closer to a partition; 
 one concludes this case as before by induction.
 
 The general case is completed exactly as before, by induction on the entries in $\lambda$, 
moving from the case where each $a(k) = 1$ to the case where each $a(k)$ is a product of factors
$(1 + x_{k,i} t)$, so is general.

J\'{o}zefiak and Pragacz \cite{JP} and Nimmo \cite{nimmo} defined a notion of skew pfaffians that similarly interpolate between pfaffians and determinants.  We do not know how to relate these two notions.

\section{A Tableau Formula for Schur-like Pfaffians}

This time we have the following proposition: 
\begin{prop}  
For $\lambda \supset \mu$ strict partitions, suppose $a = \prod_{i=1}^m \frac{1+y_i t}{1-x_i t}$, where $y_i = x_i + z$ for all $i$.  Then 
\[
 \Pf_{\lambda/\mu}(a) = \sum (x,y)^T ,
\]
the sum over all tableaux $T$ on the shifted skew shape $\lambda/\mu$  with entries 
$1' < 1 < 2' < 2 < \dots < m' < m$, weakly increasing along rows and down columns, with no $k$ 
repeated in a column and no $k'$ repeated in a row; here $(x,y)^T$ denotes the result of 
replacing each $k$ with $x_k$ and each $k'$ with $y_k$.
\end{prop}
\begin{proof}
Using the theorem of the preceding section, this follows by induction on $m$, provided we verify the case 
with $m=1$.  We can verify the special case that $\Pf_{p,q}(a) = 0$ for $p \geq q \geq 1$ directly, 
This follows from the following general formula, which holds for $a$ because $a$ satisfies the relations $a_{p,p} = 0$: 
\[ 
 a_p \, a_q = \Pf_{p,q}(a) + \sum_{i=1}^{q-1}(2 \Pf_{p+i,q-i}(a) - z\, \Pf_{p+i-1,q-i}) + 2a_{p+q} - z\,a_{p+q-1}.
\]
Using this, and induction on $q$, the vanishing of $\Pf_{p,q}(a)$ is equivalent to the equation 
$ a_p \, a_q = 2a_{p+q} - z\,a_{p+q-1}$, which follows from the fact that $a_p = (x_1+y_1)x_1^{p-1}$ and the 
fact that $y_1 = x_1 + z$.

Suppose the length of $\lambda$ is $n$.  It follows from the preceding paragraph that $\Pf_{\lambda/\mu}(a)$ vanishes whenever the length of  $\mu$ is not $n$ or $n-1$, 
since  any of the terms  $P[\sigma(\mu),0]$ will vanish.  So we are reduced to the case where $\Pf_{\lambda(a)/\mu}(a) = S_{\tilde{\lambda}/\tilde{\mu}}(a)$, 
which we have seen in Proposition~\ref{Atableaux}.
\end{proof}

\section{Difference Operators in Type C}

In (twisted) type C we work in the algebra $\bGamma[x,y]$,  with variables $x_i$ and $y_i$ for positive $i$, where $\bGamma$ is the residue ring of $\bLambda = \Z[z,c] = \Z[z,c_1,c_2, \ldots ]$ modulo 
an ideal of relations.  For this, define, for any nonnegative integers $p \geq q$, 
\[ 
 C_{p\,q} = \sum_{0 \leq i \leq j \leq q} (-1)^j \, (\tbinom{j}{i}+\tbinom{j-1}{i}) \, z^i \, c_{p-q+i} \, c_{q-j} .
\]
Set
\[
 \bGamma = \Z[z,c]/(C_{1\,1},C_{2\,2}, C_{3\,3}, \ldots ) .
\]
 
All $\Pf_\mu(c)$ vanish in $\bGamma$ if $\mu$ is a partition which is not strict, and those with $\mu$ strict form a basis for $\bGamma$ over $\Z[z]$.  So Corollary \ref{strict} applies.  We write $Q_\mu = \Pf_\mu(c)$ for this basis.

This ring  $\bGamma[x,y]$ has difference operators $\del_k^x$ and $\del_k^y$ for all  $k \geq 0$.  We take $\del_k = \del_k^x$, the case of $\del_k^y$ being obtained by interchanging each $x_i$ with $y_i$.  Again those for $k > 0$ are essentially the same as in type A, with  $\bGamma$ acting as scalars.  But $\del_0$  is different.  The involution $s_0$ interchanges $x_1$ and $z - x_1$.

One has $s_0(c) = c \cdot \frac{1+x_1}{1+z-x_1}$.  In fact, $\del_0$, 
and the related endomorphism  $s_0$ of $\bG[x]$, are determined by the formulas:
$\del_0(x_1) = -1$, $s_0(x_1) = z - x_1$, and $\del_0(c_1) = 1$.  
From the equation 
\[ 
 \del_0(f) = \frac{f - s_0(f)}{z-2x_1} ,
\]
one sees that $s_0(c_k) = c_k+(2x_1-z)\sum_{i+j=k-1}(x_1-z)^i c_j$.

The following proposition is a special case of the preceding formulas, since we know $\Pf_{\lambda/\mu}(a)$ for $a = \frac{1+x_1\, t}{1+(z-x_1)\, t}$.

\begin{prop}
For any strict partition $\lambda$, 
\[ 
 s_0(Q_\lambda(c)) = \sum \, {x_1}^{h(\lambda/\mu)} \, (x_1-z)^{v(\lambda/\mu)} \, (2x_1-z)^{k(\lambda/\mu)}  Q_\mu(c) ,
\]
the sum over all strict partitions $\mu$ obtained from $\lambda$ by removing a border 
strip (from its shifted Young diagram), with $h(\lambda/\mu)$ the number of horizontal lines 
between border boxes, $v(\lambda/\mu)$ the number of vertical lines between border 
boxes, and $k(\lambda/\mu)$ the number of connected components in the border strip.
\end{prop}

\begin{cor}
\[ 
 \del_0(Q_\lambda(c)) = \sum x_1^{h(\lambda/\mu)} \, (x_1-z)^{v(\lambda/\mu)} \, (2z_1-z)^{k(\lambda/\mu)-1}  Q_\mu(c) ,
\]
the sum over strict $\mu$ obtained from $\lambda$ by removing a non-empty border 
strip.
\end{cor}



\begin{thebibliography}{CBD2}

\bibitem[AF1]{AF1} D.~Anderson and W.~Fulton, ``Chern class formulas for classical-type degeneracy loci,'' {\em Compositio Mathematica} {\bf 154} (2018), 1746--1774.

\bibitem[AF2]{AF2} D.~Anderson and W.~Fulton, ``Schubert Polynomials in Types A and C," preprint, \texttt{arXiv:2102.05731} (2021).

\bibitem[BJS]{BJS}S.~Billey, W.~Jockusch and R.~Stanley, ``Some combinatorial properties of {S}chubert polynomials," {\em J. Algebraic Combin.} {\bf 2} (1993), 345--374.

\bibitem[CYY]{CYY}W.~Y.~C.~Chen, G.-G.~Yan and A.~L.~B.~Yang, ``The skew Schubert polynomials,'' {\em European J. Combin.} {\bf 25} (2004), 1181--1196. 

\bibitem[F]{F}W.~Fulton, ``Flags, Schubert polynomials, degeneracy loci, and determinantal formulas,'' {\em Duke Math.~J.} {\bf 65} (1992), 381--420. 

\bibitem[GV]{GV}I.~M.~Gessel and X.~G.~Viennot, ``Binomial determinants, paths, and hook length formulae," {\em Adv. in Math.} {\bf 58} ( 1985), 300--321.

\bibitem[JP]{JP}T.~J\'{o}zefiak and P.~Pragacz, ``A determinantal formula for skew {$Q$}-functions," {\em J. London Math. Soc. (2)} {\bf 43} (1991), 76--90.
     
\bibitem[KL]{KL}G.~Kempf and D.~Laksov, ``The determinantal formula of Schubert calculus,'' {\em Acta Math.} {\bf 132} (1974), 153--162.

\bibitem[LS]{LS}A.~Lascoux and M.-P.~Sch\"{u}tzenberger, ``Polyn\^{o}mes de {S}chubert," {\em  C. R. Acad. Sci. Paris S\'{e}r. I Math.} {\bf 294} (1982), 447--450.

\bibitem[LLS]{LLS}T.~Lam, S.~Lee and M.~Shimozono, ``Back Stable Schubert Calculus,'' preprint, \texttt{arXiv:1806.11233v2}, to appear in {\em Compositio Mathematica} (2021).

\bibitem[Mac1]{Mac1}I.~G.~Macdonald, {\em Symmetric functions and Hall polynomials}, second ed., Oxford University Press, 1995.

\bibitem[Mo]{Mo}I.~A.~Molev, ``Comultiplication rules for the double {S}chur functions and {C}auchy identities,'' {\em Electron. J. Combin.} {\bf 16} (2009), 1--44.

\bibitem[N]{nimmo} J.~J.~C.~Nimmo, ``Hall-Littlewood symmetric functions and the BKP equation,'' {\em J. Phys. A} {\bf 23} (1990), no. 5, 751--760.


\bibitem[W]{W}M.~L.~ Wachs, ``Flagged {S}chur functions, {S}chubert polynomials, and symmetrizing operators,''  {\em J. Combin. Theory Ser. A} {\bf 40} (1985), 276--289.
              
\end{thebibliography}
\end{document}